\documentclass{amsart}
\usepackage{verbatim}
\usepackage{rotating} 
\usepackage{amssymb}
\usepackage{mathrsfs,amsfonts,amsmath,amssymb,epsfig,amscd,xy,amsthm,pb-diagram} 

\newtheorem{theorem}{Theorem}[section]
\newtheorem{proposition}[theorem]{Proposition}

\newtheorem{lemma}[theorem]{Lemma}

\newtheorem{corollary}[theorem]{Corollary}

\newtheorem{definition}[theorem]{Definition}

\newtheorem{conjecture}[theorem]{Conjecture}

\theoremstyle{plain}

\theoremstyle{remark}

\newtheorem{remark}[theorem]{Remark}
\newtheorem{example}[theorem]{Example}

\newcommand{\C}{{\mathbb C}}

\newcommand{\Q}{{\mathbb Q}}

\newcommand{\R}{{\mathbb R}}
\newcommand{\D}{{\mathbb D}}
\newcommand{\T}{{\mathbb T}}
\newcommand{\Z}{{\mathbb Z}}

\newcommand{\fp}{\mathfrak p}

\DeclareMathOperator{\GF}{GF}

\DeclareMathOperator{\lcm}{lcm}

\DeclareMathOperator{\Span}{Span}

\DeclareMathOperator{\Norm}{N}

\DeclareMathOperator{\card}{card}

\newcommand{\bF}{{\mathbb F}}

\newcommand{\cO}{\mathcal{O}}

\author{Keira Gunn}
\address{
Keira Gunn \\
Department of Mathematics and Statistics\\
University of Calgary\\
AB T2N 1N4, Canada
}
\email{keira.gunn1@ucalgary.ca}

\author{Khoa D.~Nguyen}
\address{
Khoa D.~Nguyen \\
Department of Mathematics and Statistics\\
University of Calgary\\
AB T2N 1N4, Canada
}
\email{dangkhoa.nguyen@ucalgary.ca}

\author{J.~C.~Saunders}
\address{
J.~C.~Saunders \\
Department of Mathematics and Statistics\\
University of Calgary\\
AB T2N 1N4, Canada
}
\email{john.saunders1@ualgary.ca}


\begin{document}
	\title[Endomorphisms of positive characteristic tori]{Endomorphisms of positive characteristic tori: entropy and zeta function}
	
	\date{May 2022}
	
	\begin{abstract}
	Let $F$ be a finite field of order $q$ and characteristic $p$. Let $\mathbb{Z}_F=F[t]$, $\mathbb{Q}_F=F(t)$, $\mathbb{R}_F=F((1/t))$ equipped with the discrete valuation for which $1/t$ is a uniformizer, and let $\mathbb{T}_F=\mathbb{R}_F/\mathbb{Z}_F$ which has the structure of a compact abelian group. Let $d$ be a positive integer and let $A$ be a $d\times d$-matrix with entries in $\mathbb{Z}_F$ and non-zero determinant. The multiplication-by-$A$ map is a surjective endomorphism on $\mathbb{T}_F^d$. First, we compute the entropy of this endomorphism; the result and arguments are analogous to those for the classical case $\mathbb{T}^d=\mathbb{R}^d/\mathbb{Z}^d$. Second and most importantly, we resolve the 
	algebraicity problem for the Artin-Mazur zeta function of all such endomorphisms. As a consequence of our main result, we provide a complete characterization and an explicit formula related to the entropy when the zeta function is algebraic.   
	\end{abstract}
	
	\maketitle

	\section{Positive characteristic tori and statements of the main results}\label{sec:first}
	The tori $\T^d:=\R^d/\Z^d$ where $d$ is a positive integer play an important role in number theory, dynamical systems, and many other areas of mathematics. In this paper, we study the entropy and algebraicity of the Artin-Mazur zeta function of a surjective endomorphism on the so called positive characteristic tori.
	
	Throughout this paper, let $F$ be the finite field of order $q$ and characteristic $p$. Let
	$\Z_{F}=F[t]$ be the polynomial ring over $F$, $\Q_{F}=F(t)$, and 
	$$\R_{F}=F((1/t))=\left\{\sum_{i\leq m} a_it^i:\ m\in\Z,\ a_i\in F\ \text{for}\ i\leq m\right\}.$$
	The field $\R_{F}$ is equipped with the discrete valuation 
	$$v: \R_{F}\rightarrow \Z\cup\{\infty\}$$ 
	given by $v(0)=\infty$ and 
	$v(x)=-m$ where $x=\displaystyle\sum_{i\leq m} a_it^i$ with $a_m\neq 0$; in fact $\R_{F}$ is the completion of $\Q_{F}$ with respect to this valuation. Let $\vert\cdot\vert$ denote the non-archimedean absolute value $\vert x\vert=q^{-v(x)}$ for $x\in\R_{F}$. We fix an algebraic closure of $\R_F$ and the   absolute value $\vert\cdot\vert$ can be extended uniquely 
	to the algebraic closure (see Proposition~\ref{prop:extending absolute value}). 
	Let $\T_{F}=\R_{F}/\Z_{F}$ and let $\pi:\ \R_{F}\rightarrow \T_{F}$ be the quotient map. Every element $\alpha\in \T_{F}$ has the unique preimage $\tilde{\alpha}\in \R_{F}$ of the form
	$$\tilde{\alpha}=\sum_{i\leq -1} a_it^i.$$ 
	This yields a homeomorphism 
	$\T_{\bF}\cong \displaystyle\prod_{i\leq -1}F$ of compact abelian groups. Let $\mu$ be the probability Haar measure on $\T_{F}$ and let $\rho$ be the metric on $\T_F$ given by
	$\rho(\alpha,\beta):=\vert \tilde{\alpha}-\tilde{\beta}\vert$. We fix a positive integer
	$d$ and let $\mu^d$ be the product measure on $\T_F^d$.
	
	The analytic number theory, more specifically the theory of characters and $L$-functions, on $\T_F$ has been studied since at least 1965 in work of Hayes \cite{Hay65_TD}. Some relatively recent results include work of Liu-Wooley \cite{LW10_WP} on Waring's problem and the circle method in function fields
	and work of Porritt \cite{Por18_AN} and Bienvenu-L\^{e} \cite{BL19_LA} on correlation between the M\"obius function and a character over $\Z_F$. For a recent work in the ergodic theory side, we refer the readers to the paper by Bergelson-Leibman \cite{BL16_AW} and its reference in which the authors establish a Weyl-type equidistribution theorem.
	
	Let $A\in M_d(\Z_F)$ having non-zero discriminant. The multiplication-by-$A$ map yields a surjective endomorphism of $\T_F^d$ for which $\mu^d$ is an invariant measure, we abuse the notation by using $A$ to denote this endomorphism. Our first result is the following:
	\begin{theorem}\label{thm:entropy}
	Let $h(\mu^d,A)$ denote the entropy of $A$ with respect to $\mu^d$ and let $h(A)$ denote the topological entropy of $A$. Let $\lambda_1,\ldots,\lambda_d$ denote the eigenvalues of $A$. We have:
	$$h(A)=h(\mu^d,A)=\sum_{i=1}^d\log\max\{\vert \lambda_i\vert,1\}.$$
	\end{theorem}

	\begin{remark}
	This is the same formula as the entropy of surjective endomorphisms of $\T^d$. The proof is not surprising either: we use similar arguments to the classical ones presented in the books by Walters \cite{Wal82_AI} and Viana-Oliveira \cite{VO16_FO} together with several adaptations to the non-archimedean setting of $\R_F^d$ and $\T_F^d$. What is important is the relationship between the entropy and the Artin-Mazur zeta function in the next main result.
	\end{remark}
	
	Let $f:\ X\rightarrow X$ be  a map from a topological space $X$ to itself. For each
	$k\geq 1$, let $N_k(f)$ denote the number of \emph{isolated} fixed points of $f^k$. Assume that $N_k(f)$ is 
	finite for every $k$, then one can define the Artin-Mazur zeta function \cite{AM65_OP}:
	$$\zeta_f(z)=\exp\left(\sum_{k=1}^{\infty} \frac{N_k(f)}{k} z^k\right).$$
	When $X$ is a compact differentiable manifold and $f$ is a smooth map such that $N_k(f)$ grows at most exponentially in $k$, the  question of whether $\zeta_f(z)$ is algebraic is stated in \cite{AM65_OP}. The rationality of $\zeta_f(z)$ when $f$ is an Axiom A diffeomorphism is established by Manning \cite{Man71_AA} after earlier work by Guckenheimer \cite{Guc70_AA}. 
	On the other hand, when $X$ is an algebraic variety defined over a finite field and $f$ is the Frobenius morphism,  the function $\zeta_f(z)$ is precisely the classical zeta function of the variety $X$ and its rationality is conjectured by Weil \cite{Wei49_NO} and first established by 
	Dwork \cite{Dwo60_OT}. For the dynamics of a univariate rational function, rationality of $\zeta_f(x)$ is established by  Hinkkanen in characteristic zero \cite{Hin94_ZF} while Bridy \cite{Bri12_TO,Bri16_TA} obtains both rationality and \emph{transcendence} results over positive characteristic when $f$ belongs to certain special families of rational functions. As before, let $A\in M_d(\Z_F)$ and we use $A$ to denote the induced endomorphism on $\T_F^d$. We will show that $N_k(A)<\infty$ for every $n$ and hence one can define the zeta function $\zeta_A(z)$. As a consequence of our next main result, we resolve the algebraicity problem for $\zeta_A(z)$: we provide a complete characterization and an explicit formula when $\zeta_A(z)$ is algebraic. We need a couple of definitions before stating our result.

	Let $K$ be a finite extension of $\R_F$. Let 
	$$\cO_K:=\{\alpha\in K:\ \vert\alpha\vert\leq 1\},$$ 
	$$\cO_K^{*}=\{\alpha\in K:\ \vert\alpha\vert=1\},\ \text{and}$$
	$$\fp_K:=\{\alpha\in K:\ \vert\alpha\vert <1\}$$
	respectively denote the valuation ring, unit group, and maximal ideal. In particular: 
	$$\cO:=\cO_{\R_F}=F[[1/t]]\ \text{and}\ 
	\fp:=\fp_{\R_F}=\displaystyle\frac{1}{t} F[[1/t]]=\left\{\sum_{i\leq -1}a_i t^i:\ a_i\in F\ \forall i\right\}.$$ 
	Note that $\fp$ is the compact open subset of $\R_F$ that is both the open ball of radius $1$ and closed ball of radius $1/q$ centered at $0$. The field $\cO_K/\fp_K$ is a finite extension of 
	$\cO/\fp=F$ and the degree of this extension is called the inertia degree of $K/\R_F$ 
	\cite[p.~150]{Neu99_AN}. Let $\delta$ be this inertia degree, then $\cO_K/\fp_K$ is isomorphic to the finite field $\GF(q^{\delta})$. By applying Hensel's lemma \cite[pp.~129--131]{Neu99_AN} for the polynomial $X^{q^{\delta}-1}-1$, we have that $K$ contains all the roots of $X^{q^{\delta}-1}-1$. These roots together with $0$ form a unique copy of $\GF(q^{\delta})$ in $K$ called the Teichm\"uller representatives. This allows us to regard $\GF(q^{\delta})$ as a subfield of $K$; in fact $\GF(q^{\delta})$ is exactly the set of all the roots of unity in $K$ together with $0$. For every $\alpha\in \cO_K$, we can express uniquely:
	\begin{equation}\label{eq:alpha0 and alpha1}
	\alpha=\alpha_{(0)}+\alpha_{(1)}
	\end{equation}
	where $\alpha_{(0)}\in \GF(q^{\delta})$ and $\alpha_{(1)}\in \fp_K$.
	 
	\begin{definition}
	Let $\alpha$ be algebraic over $\R_F$ such that $\vert\alpha\vert\leq 1$. Let $K$ be a finite extension of $\R_F$ containing $\alpha$. 
	We call $\alpha_{(0)}$ and $\alpha_{(1)}$ in \eqref{eq:alpha0 and alpha1} respectively the constant term and $\fp$-term of $\alpha$; they are independent of the choice of $K$. When $\vert\alpha\vert=1$, the order of $\alpha$ modulo $\fp$ means the order of $\alpha_{(0)}$ in the multiplicative group
	$GF(q^{\delta})^*$ where $\delta$ is the inertia degree of $K/\R_F$; this is independent of the choice of $K$ as well. In fact, this order is the smallest positive integer $n$ such that $\vert \alpha^n-1\vert <1$.	
	\end{definition}
	
	We identify the rational functions in $\C(z)$ to the corresponding Laurent series in $\C((z))$. 
\begin{definition}	
	A series $f(z)\in \C((z))$ is called D-finite if all of its formal derivatives $f^{(n)}(z)$ for $n=0,1,\ldots$ span a finite dimensional vectors space over $\C(z)$. Equivalently, there 
	exist an integer $n\geq 0$ and $a_0(z),\ldots,a_{n}(z)\in \C[z]$ with $a_{n}\neq 0$ such that:
	$$a_n(z)f^{(n)}(z)+a_{n-1}f^{(n-1)}(z)+\ldots+a_0(z)f(z)=0.$$ 
\end{definition}

\begin{remark}\label{rem:algebraic implies D-finite}
Suppose that $f(z)\in\C[[z]]$ is algebraic then $f$ is D-finite, see
\cite[Theorem~2.1]{Sta80_DF}.
\end{remark}
	
	Our next main result is the following:
	\begin{theorem}\label{thm:zeta}
	Let $A\in M_d(\Z_F)$ and put 
	$\displaystyle r(A)=\prod_{\lambda}\max\{1,\vert\lambda\vert\}$
	where $\lambda$ ranges over all the $d$ eigenvalues of $A$; we have
	$r(A)=e^{h(A)}$ when $\det(A)\neq 0$ thanks to Theorem~\ref{thm:entropy}. Among the $d$ eigenvalues of $A$, let $\mu_1,\ldots,\mu_M$ be all the eigenvalues that are roots of unity and let $\eta_1,\ldots,\eta_N$ be all the eigenvalues that have absolute value $1$ and are not roots of unity. For $1\leq i\leq M$, let $m_i$ denote the order of $\mu_i$ modulo $\fp$. For $1\leq i\leq N$, let $n_i$ denote the order of $\eta_i$ modulo $\fp$.
	We have:
	\begin{itemize}
	\item [(a)] Suppose that for every $j\in \{1,\ldots,N\}$, there exists $i\in\{1,\ldots,M\}$ such that
	$m_i\mid n_j$. Then $\zeta_A(z)$ is algebraic and 
	$$\zeta_A(z)=(1-r(A)z)^{-1}\prod_{1\leq \ell\leq M} \prod_{1\leq i_1<i_2<\ldots<i_{\ell}\leq M}R_{A,i_1,\ldots,i_{\ell}}(z)$$
	where $\displaystyle R_{A,i_1,\ldots,i_{\ell}}(z):=\left(1-\left(r(A)z\right)^{\lcm(m_{i_1},\ldots,m_{i_\ell})}\right)^{(-1)^{\ell+1}/\lcm(m_{i_1},\ldots,m_{i_\ell})}$.

	\item [(b)] Otherwise suppose there exists $j\in\{1,\ldots,N\}$ such that for every $i\in\{1,\ldots,M\}$, we have $m_i\nmid n_j$. Then the series $\displaystyle\sum_{k=1}^{\infty}N_k(A)z^k$ converges in the open disk
	$\{z\in\C:\ \vert z\vert < 1/r(A)\}$ and it is not $D$-finite.  
	 Consequently, the function $\zeta_A(z)$ is transcendental.
	\end{itemize}
	\end{theorem}
	
	\begin{remark}
	We allow the possibility that any (or even both) of $M$ and $N$ to be $0$. When $N=0$, the condition in (a) is vacuously true and $\zeta_A(z)$ is algebraic in this case. When $N=0$ and $M=0$ meaning that none of the eigenvalues of $A$ has absolute value $1$, the product
	$\displaystyle\prod_{1\leq j\leq M}$ in (a) is the empty product and $\zeta_A(z)=\displaystyle\frac{1}{1-r(A)z}$. When $M=0$ and $N>0$, the condition in (b) is vacuously true and $\zeta_A(z)$ is transcendental in this case.
	\end{remark}
	
	Our results are quite different from results in work of Baake-Lau-Paskunas \cite{BLP10_AN}. In \cite{BLP10_AN}, the authors prove that the zeta function of endomorphisms of the classical tori $\T^d$ are always rational. In our setting, we have cases when the zeta function is rational, transcendental, or algebraic irrational:
	
	\begin{example}\label{eg:algebraic irrational}
	Let $F=\GF(7)$ and let $A$ be the diagonal matrix with diagonal entries $\alpha,\beta\in \GF(7)^*$ where $\alpha$ has order $2$ and $\beta$ has order $3$. Then
	$$\zeta_A(z)=\frac{(1-z^2)^{1/2}(1-z^3)^{1/3}}{(1-z)(1-z^6)^{1/6}}$$
	is algebraic irrational.
	\end{example}

	In work of Bell-Miles-Ward \cite{BMW14_TA}, the authors conjecture and obtain some partial results concerning the following P\'olya-Carlson type dichotomy \cite{Car21_U,Pol28_U} for a \emph{slightly different} zeta function: it is either rational or admits a natural boundary at its radius of convergence.
	\begin{conjecture}[Bell-Miles-Ward, 2014]\label{conj:BMW}
	Let $\theta:\ X\rightarrow X$ be an automorphism of  a compact metric abelian group with the property that $\tilde{N}_{k}(\theta)<\infty$ for every $k\geq 1$ where $\tilde{N}_k(\theta)$ denotes the number of fixed points of $\theta^k$. Then 
	$$\tilde{\zeta}_{\theta}(z):=\exp\left(\sum_{k=1}^{\infty} \frac{\tilde{N}_k(\theta)}{k} z^k\right)$$ 
	is either a rational function or admits a natural boundary.
	\end{conjecture}
	
	\begin{remark}\label{rem:BMW vs us}
	The difference between $\tilde{\zeta}_{\theta}$ in \ref{conj:BMW} and the Artin-Mazur zeta function $\zeta_f$ is that the latter involves the number of isolate fixed points. Example~\ref{eg:algebraic irrational} is not included in Conjecture~\ref{conj:BMW} since $A^6$ is the identity matrix and hence $\tilde{N}_6(A)=\infty$
	while we have $N_6(A)=0$ (see Lemma~\ref{lem:N_k(A) formula}). When $A\in M_d(\Z_F)$ has the property that none of its eigenvalues is a root of unity, 
one can show that $N_k(A)=\tilde{N}_k(A)$ and hence $\zeta_A(z)=\tilde{\zeta}_A(z)$.
Conjecture~\ref{conj:BMW} predicts that when $M=0$ and $N>0$ in Theorem~\ref{thm:zeta}, the zeta function $\zeta_A(z)=\tilde{\zeta}_A(z)$ admits the circle of radius $1/r(A)$ as a natural boundary. We can only prove this in some special cases and leave it for future work.	
	\end{remark}

	For the proof of Theorem~\ref{thm:zeta}, we first derive a formula for $N_k(A)$ and it turns out that one needs to study $\vert\lambda^k-1\vert$ where $\lambda$ is an eigenvalue of $A$. When $\vert\lambda\vert\neq 1$, one immediately has
	$\vert\lambda^k-1\vert=\max\{1,\vert\lambda\vert\}^k$. However, when $\vert\lambda\vert=1$ (i.e.
	$\lambda$ is among the $\mu_i$'s and $\eta_j$'s), a more refined analysis is necessary to study
	$\vert\lambda^k-1\vert$. After that, part (a) can be proved by a direct computation. 
	On the other hand, the proof of part (b) is more intricate. We first assume that the series $\displaystyle\sum_{k=1}^{\infty}N_k(A)z^k$ is D-finite, then use a certain linear recurrence relation satisfied by D-finite power series to contradict the peculiar value of $N_k(A)$ at certain $k$.

	\textbf{Acknowledgements.}
	The first author is partially supported by a Vanier Canada Graduate Scholarship. The second and third authors are partially supported by an NSERC Discovery Grant and a CRC Research Stipend. We are grateful to Professors Jason Bell, Michael Singer, and Tom Ward for useful comments that help improve the paper.
	 
	\textbf{Notes added in May 2022.} This paper is superseded by \cite{BGNS22_AG} by Bell and the authors and no longer intended for publication. Inspired by the earlier work \cite{BNZ20_DF,BNZ22_DF2}, the paper \cite{BGNS22_AG} establishes a general
	P\'olya-Carlson criterion and applies this to confirm that the zeta function
	$\zeta_A(z)$ admits the circle of radius $1/r(A)$ as a natural boundary in the transcendence case (see Remark~\ref{rem:BMW vs us}).

	\section{Normed vector spaces and linear maps}
	Throughout this section, let $K$ be a field that is complete with respect to a nontrivial absolute value $\vert\cdot\vert$; nontriviality means that there exists $x\in K^{*}$ such that $\vert x\vert\neq 1$. We have:
	\begin{proposition}\label{prop:extending absolute value}
	Let $E/K$ be a finite extension of degree $n$. Then $\vert\cdot\vert$ can be extended in a unique way
	to an absolute value on $E$ and this extension is given by the formula:
	$$\vert \alpha\vert=\vert \Norm_{E/K}(\alpha)\vert^{1/n}\ \text{for every $\alpha\in E$.}$$		
	The field $E$ is complete with respect to this extended absolute value.
	\end{proposition}
	\begin{proof}
	See \cite[pp.~131--132]{Neu99_AN}.
	\end{proof}
	We now fix an algebraic closure of $K$ and extend $\vert\cdot\vert$ to an absolute value on this algebraic closure thanks to Proposition~\ref{prop:extending absolute value}. For a vector space $V$ over $K$, a norm on $V$ is a function $\Vert\cdot\Vert:\ V\rightarrow\R_{\geq 0}$ such that:
	\begin{itemize}
		\item $\Vert x\Vert=0$ iff $x=0$.
		\item $\Vert cx\Vert=\vert c\vert\cdot \Vert x\Vert$ for every $c\in K$ and $v\in V$.
		\item $\Vert x+y\Vert\leq \Vert x\Vert+\Vert y\Vert$ for every $x,y\in V$.	
	\end{itemize}
	Two norms $\Vert\cdot\Vert$ and $\Vert\cdot\Vert'$ on $V$ are said to be equivalent if there exists a positive constant $C$ such that 
	$$\frac{1}{C} \Vert x\Vert \leq \Vert x\Vert'\leq C\Vert x\Vert$$ 
	for every $x\in V$. It is well-known that any two norms on a finite dimensional vector space $V$
	are equivalent to each other and $V$ is complete with respect to any norm, see \cite[pp.~132--133]{Neu99_AN}.
	
	\begin{proposition}\label{prop:1 Jordan block}
	Let $V$ be a vector space over $K$ of finite dimension $d>0$. Let $\ell:\ V\rightarrow V$ be an 
	invertible $K$-linear map such that there exist $\lambda\in K^*$ and 
	a basis $x_1,\ldots,x_d$ of $V$ over $K$ with:
	$$\ell(x_1)=\lambda x_1\ \text{and}\ \ell(x_i)=\lambda x_i+x_{i-1}\ \text{for $2\leq i\leq d$};$$ 
	in other words, the matrix of $\ell$ with respect to $x_1,\ldots,x_d$ is one single Jordan block
	with eigenvalue $\lambda$. Let $\delta>0$. Then there exists a norm
	$\Vert\cdot\Vert$ on $V$ such that:
	\begin{equation}\label{eq:1 Jordan block}
	(1-\delta)\vert\lambda\vert\cdot \Vert x\Vert\leq \Vert\ell(x)\Vert\leq (1+\delta)\vert\lambda\vert\cdot\Vert x\Vert
	\end{equation}
	for every $x\in V$.		
	\end{proposition}
	\begin{proof}
	We proceed by induction on $d$. The case $d=1$ is obvious since we can take $\Vert\cdot\Vert$ to be 
	any norm and we have $\Vert \ell(x_1)\Vert=\vert\lambda\vert\Vert x_1\Vert$. Let $d\geq 2$ and 
	suppose 
	the proposition holds for any vector space of dimension at most $d-1$. Let $V'=\Span(x_1,\ldots,x_{d-1})$. By the induction hypothesis, there exists a norm
	$\Vert\cdot\Vert'$ on $V'$ such that
	\begin{equation}\label{eq:1 Jordan block ell and norm on V'}
	(1-\delta)\vert\lambda\vert \cdot \Vert x'\Vert'\leq \Vert \ell(x')\Vert' \leq (1+\delta)\vert \lambda\vert \cdot \Vert x'\Vert'
	\end{equation}
	for every $x'\in V'$.

	Let $M$ be a positive number such that:
	\begin{equation}\label{eq:1 Jordan block condition on M}
	\delta \vert \lambda\vert M\geq \Vert x_{d-1}\Vert'.
	\end{equation}
	Every $x\in V$ can be written uniquely as $x=ax_d+x'$ where $a\in K$ and $x'\in V'$, then we define the norm $\Vert\cdot\Vert$ on $V$ by the formula:
	$$\Vert x\Vert=\vert a\vert M+\Vert x'\Vert'.$$
	Note that $\ell(x)=a\lambda x_d+ax_{d-1}+\ell(x')$ and 
	$\Vert \ell(x)\Vert=\vert \lambda\vert\vert a\vert M+\Vert \ell(x')+ax_{d-1}\Vert'$. Therefore:
	\begin{align*}
	\Vert \ell(x)\Vert &\geq \vert\lambda\vert\vert a\vert M+\Vert \ell(x')\Vert'-\vert a\vert\cdot \Vert x_{d-1}\Vert'\\
	&\geq (1-\delta)\vert\lambda\vert \vert a\vert M+(1-\delta)\vert\lambda\vert\cdot\Vert x'\Vert'=(1-\delta)\vert\lambda\vert\cdot\Vert x\Vert	
	\end{align*}
	where the last inequality follows from \eqref{eq:1 Jordan block ell and norm on V'} and \eqref{eq:1 Jordan block condition on M}.
	The desired upper bound on $\Vert\ell(x)\Vert$ is obtained in a similar way:
	\begin{align*}
	\Vert \ell(x)\Vert &\leq \vert\lambda\vert\vert a\vert M+\Vert \ell(x')\Vert'+\vert a\vert\cdot\Vert x_{d-1}\Vert'\\
	&\leq (1+\delta)\vert\lambda\vert \vert a\vert M+(1+\delta)\vert\lambda\vert\cdot\Vert x'\Vert'=(1+\delta)\vert\lambda\vert\cdot\Vert x\Vert	
	\end{align*}
	and we finish the proof.
	\end{proof}

	\begin{proposition}\label{prop:prime power char pol}
	Let $V$ be a vector space over $K$ of finite dimension $d>0$. 	
	Let $\ell:\ V\rightarrow V$ be an 
	invertible $K$-linear map such that 
	the characteristic polynomial $P(X)$ of $\ell$ is the power of an irreducible polynomial in $K[X]$.
    By Proposition~\ref{prop:extending absolute value},  all the roots
    of $P$ have the same absolute value denoted by $\theta$.
	Let $\delta>0$. Then there exists a norm $\Vert\cdot\Vert$ on $V$ such that
	$$(1-\delta)\theta \Vert x\Vert\leq \Vert \ell(x)\Vert\leq (1+\delta)\theta\Vert x\Vert$$
	for every $x\in V$.	
	\end{proposition}
	\begin{proof}
	Let $E$ be the splitting field of $P(X)$ over $K$. 
	Let $V_E=E\otimes_K V$ and we still use $\ell$ to denote the induced linear operator on $V_E$.
	In the Jordan canonical form of $\ell$, let $s$ denote the number of Jordan blocks. 
	Then we have a basis $x_{1,1},\ldots,x_{1,d_1},\ldots,x_{s,1},\ldots,x_{s,d_s}$ of
	$V_E$ over $E$ such that for each $1\leq i\leq s$, the map $\ell$
	maps $V_{E,i}:=\Span_E(x_{i,1},\ldots,x_{i,d_i})$ to itself and 
	the matrix representation of $\ell$ with respect to $x_{i,1},\ldots,x_{i,d_i}$ is
	the $i$-th Jordan block. By Proposition~\ref{prop:1 Jordan block}, there exists a norm $\Vert\cdot\Vert_i$ on $V_{E,i}$ such that
	$$(1-\delta)\theta\Vert x\Vert_i\leq\Vert\ell(x)\Vert_i\leq (1+\delta)\theta\Vert x\Vert_i$$
	for every $x\in V_{E,i}$. We can now define $\Vert\cdot\Vert$ on $V_E= V_{E,1}\oplus\cdots\oplus V_{E,s}$ as $\Vert\cdot\Vert_1+\cdots+\Vert\cdot\Vert_s$. Then the restriction of $\Vert\cdot\Vert$ on
	$V$ is the desired norm.
	\end{proof}
	
	\begin{corollary}\label{cor:general ell}
	Let $V$ be a vector space over $K$ of finite dimension $d>0$. 	
	Let $\ell:\ V\rightarrow V$ be an 
	invertible $K$-linear map. Then there exist a positive integer $s$, 
	subspaces
	$V_1,\ldots,V_s$ of $V$, and positive numbers $\theta_1,\ldots,\theta_s$ with the
	following properties:
	\begin{itemize}
		\item [(i)] $\ell(V_i)\subseteq V_i$ for $1\leq i\leq s$ and $V=V_1\oplus\cdots\oplus V_s$.
		
		\item [(ii)] The multiset 
		$$\{\vert \lambda\vert:\ \text{eigenvalues $\lambda$ of $V$ counted with multiplicities}\}$$
		of order $d$ is equal to the multiset
		$$\{\theta_1,\ldots,\theta_1,\theta_2,\ldots,\theta_2,\ldots,\theta_s,\ldots,\theta_s\}$$
		in which the number of times $\theta_i$ appears is $\dim(V_i)$ for $1\leq i\leq s$.

		\item [(iii)] For every $\delta>0$, for $1\leq i\leq s$, there exists a norm $\Vert\cdot\Vert_i$ 
		on
		$V_i$ such that 
		$$(1-\delta)\theta_i\Vert x\Vert_i\leq \Vert\ell(x)\Vert_i\leq (1+\delta)\theta_i\Vert x\Vert_i$$
		for every $x\in V_i$.
	\end{itemize}
	\end{corollary}
	\begin{proof}
	By \cite[p.~424]{DF04_AA}, there exist $\ell$-invariant subspaces
	$V_1,\ldots,V_s$ of $V$ such that $V=V_1\oplus\cdots\oplus V_s$ and for $1\leq i\leq s$, the characteristic polynomial $P_i$ of the restriction of $\ell$ to $V_i$ is a power of an irreducible factor over $K$ of the characteristic polynomial of $\ell$. Let $\theta_i$ denote the common absolute value of the roots of $P_i$. Then we apply Proposition~\ref{prop:prime power char pol} and finish the proof.
	\end{proof}

	\section{The proof of Theorem~\ref{thm:entropy}}
	Recall from Section~\ref{sec:first} that
	$\pi:\ \R_F\rightarrow \T_F$ denotes the quotient map, 
	$$\fp:=\fp_{\R_F}=\displaystyle\frac{1}{t} F[[1/t]]=\left\{\sum_{i\leq -1}a_i t^i:\ a_i\in F\ \forall i\right\},$$  
	every element $\alpha\in \T_{F}$ has the unique preimage $\tilde{\alpha}\in \R_{F}$ of the form
	$$\tilde{\alpha}=\sum_{i\leq -1} a_it^i\in \fp,$$
	$\mu$ denotes the probability Haar measure on $\T_F$, and $\rho$ is the metric on 
	$\T_F$ given by
	$\rho(\alpha,\beta)=\vert \tilde{\alpha}-\tilde{\beta}\vert$.
	Let $\tilde{\mu}$ be the Haar measure on $\R_{F}$ normalized so that 
	$\tilde{\mu}(\D_{\bF})=1$. Therefore, we have that $\D_{F}$ and $\T_{F}$ are isometric as metric 
	spaces and isomorphic as probability spaces.	
	
	Let $d$ be a positive integer. On $\T_{F}^d$ and $\R_{F}^d$ we have the respective product 
	measures
	$\mu^d$ and $\tilde{\mu}^d$. Let $\vert\cdot\vert_{(d)}$ be the norm on $\R_{F}^d$ given by:
	$$\vert(x_1,\ldots,x_d)\vert_{(d)}=\max_{1\leq i\leq d} \vert x_i\vert.$$
	Then the induced metric $\rho_{(d)}$ on $\T_{F}^d$ is:
	$$\rho_{(d)}((\alpha_1,\ldots,\alpha_d),(\beta_1,\ldots,\beta_d))=\max_{1\leq i\leq d}\vert\tilde{\alpha}_i-\tilde{\beta}_i\vert.$$

	\begin{proposition}\label{prop:measure of balls}
	Let $V$ be a vector space over $\R_F$ of dimension $d$.
	Let $\Vert\cdot\Vert$ be a norm on $V$ and let $\eta$ be a Haar measure on $V$. There exist positive constants $C_1$ and $C_2$ such that the open ball
	$$B(r^{-}):=\{x\in V:\ \Vert x\Vert<r\}$$
	and the closed ball
	$$B(r):=\{x\in V:\ \Vert x\Vert\leq r\}$$
	satisfy
	$$C_1 r^d<\eta(B(r^{-})),\eta(B(r))<C_2r^d$$
	for every $r>0$.
	\end{proposition}
	\begin{proof}
	After choosing a basis, we may identify $V$ as $\R_F^d$; recall the norm $\vert\cdot\vert_{(d)}$ above. 
	By uniqueness up to scaling of Haar measures, we may assume that $\eta$ is the Haar measure normalized so that the set
	$$B':=\{(x_1,\ldots,x_d)\in\R_F^d:\ \vert(x_1,\ldots,x_d)\vert_{(d)}=\max_{1\leq i\leq d}\vert x_i\vert \leq 1\}$$
	has $\eta(B')=1$.
	
	Since $\Vert\cdot\Vert$ and $\vert\cdot\vert_{(d)}$ are equivalent to each other, there exist positive $C_3$ and $C_4$ such that both $B(r^{-})$ and $B(r)$ contain 
	$$B'(C_3r):=\{(x_1,\ldots,x_d)\in\R_F^d:\ \vert(x_1,\ldots,x_d)\vert_{(d)}=\max_{1\leq i\leq d} \vert x_i\vert \leq C_3r\}$$
	and are contained in
	$$B'(C_4r)=\{(x_1,\ldots,x_d)\in \R_F^d:\ \vert(x_1,\ldots,x_d)\vert_{(d)}=\max_{1\leq i\leq d} \vert x_i\vert \leq C_4r\}.$$

	Let $q^m$ (respectively $q^n$) be the largest (respectively smallest) power of $q$
	that is smaller than $C_3r$ (respectively larger than $C_4r$). Then we have:
	$$\eta(B'(C_3r))\geq q^{md}> (C_3r/q)^d\ \text{and}$$
	$$\eta(B'(C_4r))\leq q^{nd}<(C_4qr)^d.$$
	This finishes the proof. 
	\end{proof}

	We apply Corollary~\ref{cor:general ell} for the vector space $\R_F^d$ and the multiplication-by-$A$ map to get the invariant subspaces $V_1,\ldots,V_s$ and 
	positive numbers $\theta_1,\ldots,\theta_s$. 
	Fix a Haar measure $\eta_i$ on $V_i$ and let $\eta:=\eta_1\times\cdots\times\eta_s$ which is a Haar measure on $\R_F^d$. Let $c>0$ such that $\tilde{\mu}^d=c\eta$.

	Fix $\delta>0$, we assume that $\delta$ is sufficiently small so that 
	$(1+\delta)\theta_i<1$ whenever $\theta_i<1$. For $1\leq i\leq s$, let $\Vert\cdot\Vert_i$ be a norm on $V_i$ as given in Corollary~\ref{cor:general ell}.
	Every $x\in\R_F^d$ can be written uniquely as $x=x_1+\ldots+x_s$ with $x_i\in V_i$ for $1\leq i\leq s$ and we define the norm $\Vert\cdot\Vert$ on $\R_F^d$ by the formula:
	$$\Vert x\Vert=\max_{1\leq i\leq s}\Vert x_i\Vert_i.$$
	Since $\vert\cdot\vert_{(d)}$ and $\Vert\cdot\Vert$ are equivalent to each other, the induced metric 
	$\tau$ on $\T_F^d$ given by:
	$$\tau((\alpha_1,\ldots,\alpha_d),(\beta_1,\ldots,\beta_d)):=\Vert (\tilde{\alpha}_1-\tilde{\beta_1},\ldots,\tilde{\alpha}_d-\tilde{\beta}_d)\Vert$$
	is equivalent to $\rho_{(d)}$. 
	
	\begin{lemma}\label{lem:replacing torus metric by vector space norm}
	We still use $\pi$ to denote the quotient map
	$\R_F^d\rightarrow\T_F^d$. There exists a positive constant $C_5$ such that the following hold.
	\begin{itemize}
		\item [(i)] For any $x\in \fp^d$ and $y\in \R_F^d$, if $\Vert x-y\Vert\leq C_5$ then
		$y\in \fp^d$.
		
		\item [(ii)] For any $x,y\in\R_F^d$ such that $\Vert x-y\Vert\leq C_5$
	and $\tau(\pi(Ax),\pi(Ay))\leq C_5$, we have 
	$\tau(\pi(Ax),\pi(Ay))=\Vert Ax-Ay\Vert$.
	\end{itemize}
	
	\end{lemma}
	\begin{proof}
	For part (i), we can characterize the set $\fp^d$ as the set of $x\in\R_F^d$ such that
	$\vert x\vert_{(d)}\leq 1/q$. Hence when $\Vert x-y\Vert$ is sufficiently small, we have that 
	$\vert x-y\vert_{(d)}\leq 1/q$ thanks to equivalence of these norms. Hence $x-y\in \fp^d$ and we
	have $y\in\fp^d$.

	We now consider part (ii). Since $\vert z\vert_{(d)}\geq 1$ for every non-zero $z\in\Z_F^d$ and since $\Vert\cdot\Vert$ and $\vert\cdot\vert_{(d)}$ are equivalent, there is a positive constant  $C_6$ such that
	$\Vert z\Vert\geq C_6$ for every non-zero $z\in \Z_F^d$.
	
	There exists $C_7$ such that 
	$\Vert Aw\Vert \leq C_7\Vert w\Vert$ for every $w\in \R_F^d$; for instance we may take 
	$C_7=(1+\delta)\max_{1\leq i\leq s} \theta_i$ thanks to the definition
	of $\Vert\cdot\Vert$ and properties
	of the $\Vert\cdot\Vert_i$'s in Corollary~\ref{cor:general ell}. 
	
	We now choose $C_5$ to be any positive constant such that $C_5<\frac{C_6}{C_7+1}$. Let $x,y\in \R_F^d$ satisfying conditions in the statement of the lemma. 
	We have
	$$C_5\geq \tau(\pi(Ax),\pi(Ay))=\Vert Ax-Ay+z\Vert$$
	for some $z\in\Z_F^d$. If $z\neq 0$ then we have 
	$$C_7C_5\geq C_7\Vert x-y\Vert\geq \Vert Ax-Ay\Vert\geq \Vert z\Vert - \Vert Ax-Ay+z\Vert\geq C_6-C_5,$$
	contradicting the choice of $C_5$. Hence $z=0$ and we are done.
	\end{proof}
	
	\begin{proof}[Proof of Theorem~\ref{thm:entropy}]
	 Let $\alpha=(\alpha_1,\ldots,\alpha_d)\in \T_F^d$ and let $x=(\tilde{\alpha}_1,\ldots,\tilde{\alpha}_d)$ which is the preimage of $\alpha$ in $\fp^d$. Let 
	$\epsilon>0$ and $n\geq 1$. 
	All the implicit constants below might depend on the choice of the norms $\Vert\cdot\Vert_i$'s hence depending on $\delta$ but they are independent of $\epsilon$ and $n$.
	
	Let 
	$$B(\alpha,\epsilon,n):=\{\beta=(\beta_1,\ldots,\beta_d)\in \T_F^d:\ \rho_{(d)}(A^j\alpha,A^j\beta)<\epsilon\ \text{for  $j=0,1,\ldots,n-1$}\}.$$
	We aim to obtain an upper bound on $\mu^d(B(\alpha,\epsilon,n))$. Thanks to equivalence between $\rho_{(d)}$ and $\tau$, there exists a positive constant $C_8$ such that $B(\alpha,\epsilon,n)$ is contained in
	$$B'(\alpha,\epsilon,n):=\{\beta=(\beta_1,\ldots,\beta_d)\in \T_F^d:\ \tau(A^j\alpha,A^j\beta)<C_8\epsilon\ \text{for  $j=0,1,\ldots,n-1$}\}.$$
	
	For $\beta=(\beta_1,\ldots,\beta_d)\in B'(\alpha,\epsilon,n)$, let $y=(\tilde{\beta}_1,\ldots,\tilde{\beta}_d)$ and we have
	$\Vert x-y\Vert=\tau(\alpha,\beta)<C_8\epsilon$. When $\epsilon$ is sufficiently small so that 
	$C_8\epsilon$ is smaller than the constant $C_5$ in
	Lemma~\ref{lem:replacing torus metric by vector space norm}, we can apply this lemma repeatedly 
	to get 
	$$B'(\alpha,\epsilon,n)=\{\pi(y):\ y\in \fp^d\ \text{and}\ \Vert A^jx-A^jy\Vert<C_8\epsilon\ \text{for $j=0,1,\ldots,n-1$}\}.$$
	 
	 By Lemma~\ref{lem:replacing torus metric by vector space norm}, the condition
	 $y\in\fp^d$ is automatic once we have $\Vert x-y\Vert<C_8\epsilon<C_5$ and $x\in\fp^d$.
	 Let 
	 $$\tilde{B}'(x,\epsilon,n):=\{y\in \R_F^d:\ \Vert A^jx-A^jy\Vert<C_8\epsilon\ \text{for $j=0,1,\ldots,n-1$}\},$$ 
	 we have 
	 $\mu^d(B'(\alpha,\epsilon,n))=\tilde{\mu}^d(\tilde{B}'(x,\epsilon,n))=c\eta(\tilde{B}'(x,\epsilon,n)).$
	
	We express $x=x_1+\ldots+x_s$ and $y=y_1+\ldots+y_s$ where each $x_i,y_i\in V_i$.
	The condition in the description of $\tilde{B}'(x,\epsilon,n)$ is equivalent to
	$\Vert x_i-y_i\Vert_i<C_8\epsilon$ and 
	$\Vert A^jx_i-A^jy_i\Vert_i<C_8\epsilon$ for every $1\leq i\leq s$ and $1\leq j\leq n-1$.	
	We use Corollary~\ref{cor:general ell} to have:
	\begin{equation}\label{eq:use Cor general ell for A^j}
	((1-\delta)\theta_i)^j \Vert x_i-y_i\Vert_i \leq \Vert A^jx_i-A^jy_i\Vert_i\leq ((1+\delta)\theta_i)^j\Vert x_i-y_i\Vert_i.
	\end{equation}
	Let $I=\{i\in\{1,\ldots,s\}:\ \theta_i\geq 1\}$ and since we choose $\delta$ sufficiently small
	so that $(1+\delta)\theta_i<1$ whenever $\theta_i<1$, inequality \eqref{eq:use Cor general ell for A^j} implies that the set $\tilde{B}'(x,\epsilon,n)$
	is contained in the set:
	
	\begin{align*}
	\{y=y_1+\ldots+y_s:\ & \Vert x_i-y_i\Vert_i<C_8\epsilon ((1-\delta)\theta_i)^{-(n-1)}\ \text{for $i\in I$}\\
	&\text{and}\ \Vert x_i-y_i\Vert_i<C_8\epsilon\ \text{for $i\notin I$}\}.
	\end{align*}
	
	Let $d_i=\dim(V_i)$ for $1\leq i\leq s$. By Proposition~\ref{prop:measure of balls}, there exists a constant $C_9$ such that:
	\begin{equation}\label{eq:C9}
	\mu^d(B'(\alpha,\epsilon,n))=c\eta(\tilde{B}'(x,\epsilon,n))<C_9 \prod_{i\in I}(C_8\epsilon)^{d_i}((1-\delta)\theta_i)^{-d_i(n-1)}.
	\end{equation}
	Put $h^{+}(\mu^d,A,x,\epsilon)= \displaystyle\limsup_{n\to\infty} \frac{-\log(\mu^d(B(\alpha,\epsilon,n)))}{n}$, then \eqref{eq:C9} implies:
	$$\sum_{i\in I}d_i\log(1-\delta)+\sum_{i\in I}d_i\log\theta_i\leq h^{+}(\mu,A,x,\epsilon).$$
	Recall that our only assumption on $\epsilon$ is that it is sufficiently small so that $C_8\epsilon<C_5$.

	For the other inequality, we argue in a similar way. There exists a constant $C_{10}$ such that set $B(\alpha,\epsilon,n)$ contains the set:
	$$B''(\alpha,\epsilon,n):=\{\beta=(\beta_1,\ldots,\beta_d)\in \T_F^d:\ \tau(A^j\alpha,A^j\beta)<C_{10}\epsilon\ \text{for  $0\leq j\leq n-1$}\}.$$
    And when $\epsilon$ is sufficiently small so that $C_{10}\epsilon<C_5$, we apply Lemma~\ref{lem:replacing torus metric by vector space norm} repeatedly to get
    $$B''(\alpha,\epsilon,n)=\{\pi(y):\ y\in \fp^d\ \text{and}\ \Vert A^jx-A^jy\Vert<C_{10}\epsilon\ \text{for $j=0,1,\ldots,n-1$}\}.$$
    Then consider
    $$\tilde{B}''(x,\epsilon,n):=\{y\in \R_F^d:\ \Vert A^jx-A^jy\Vert<C_{10}\epsilon\ \text{for $j=0,1,\ldots,n-1$}\},$$ 
	we have 
	$\mu^d(B''(\alpha,\epsilon,n))=\tilde{\mu}^d(\tilde{B}''(x,\epsilon,n))=c\eta(\tilde{B}''(x,\epsilon,n)).$	Arguing as before, the set 
	$\tilde{B}''(x,\epsilon,n)$
	contains the set:
	\begin{align*}
	\{y=y_1+\ldots+y_s:\ & \Vert x_i-y_i\Vert_i<C_{10}\epsilon ((1+\delta)\theta_i)^{-(n-1)}\ \text{for $i\in I$}\\
	&\text{and}\ \Vert x_i-y_i\Vert_i<C_{10}\epsilon\ \text{for $i\notin I$}\}.
	\end{align*}
	Then we can use Proposition~\ref{prop:measure of balls} to get a constant $C_{11}$ such that:
	$$C_{11}\prod_{i\in I}(C_{10}\epsilon)^{d_i}((1+\delta)\theta_i)^{-d_i(n-1)}<\eta(\tilde{B}''(x,\epsilon,n)).$$
	This implies
	$$h^{+}(\mu,A,x,\epsilon)\leq \sum_{i\in I}d_i\log(1+\delta)+\sum_{i\in I}d_i\log\theta_i$$
	when $\epsilon$ is sufficiently small.
	
	Therefore
	$$\sum_{i\in I}d_i\log(1-\delta)+\sum_{i\in I}d_i\log\theta_i\leq \lim_{\epsilon\to 0^{+}} h^{+}(\mu,A,x,\epsilon)\leq \sum_{i\in I}d_i\log(1+\delta)+\sum_{i\in I}d_i\log\theta_i.$$
	Since $\delta$ can be arbitrarily small, we conclude that
	$$\lim_{\epsilon\to 0^{+}} h^{+}(\mu,A,x,\epsilon)=\sum_{i\in I}d_i\log\theta_i=\sum_{i=1}^d\log\max\{\vert\lambda_i\vert,1\}$$
	where the last equality follows from Property~(ii) in Corollary~\ref{cor:general ell}. By the Brin-Katok theorem (see \cite{BK83_OL} and \cite[pp.~262--263]{VO16_FO}), we have:
	$$h(\mu^d,A)=\sum_{i=1}^d\log\max\{\vert\lambda_i\vert,1\}.$$
	It is well-known that $h(A)=h(\mu^d,A)$ \cite[p.~197]{Wal82_AI} and this finishes the proof.
	\end{proof}

	\section{The proof of Theorem~\ref{thm:zeta}}
	Throughout this section, we assume the notation in the statement of Theorem~\ref{thm:zeta}. 
	Let $I$ denote the identity matrix in $M_d(\Z_F)$. The below formula for $N_k(A)$ in the classical case is well-known \cite{BLP10_AN}:
	\begin{lemma}\label{lem:N_k(A) formula}
	Let $B\in M_d(\Z_F)$. The number of isolated fixed points $N_1(B)$ of the multiplication-by-$B$ map
	$$B:\ \T_F^d\rightarrow \T_F^d$$
	is $\vert \det(B-I)\vert$. Consequently $N_k(A)=\vert\det(A^k-I)\vert$ for every $k\geq 1$.
	\end{lemma}
	\begin{proof}
	When $\det(B-I)=0$, there is a non-zero $x\in\R_F^d$ such that $Bx=x$. Then for any 
	fixed point $y\in \T_F^d$, the points $y+cx$ for $c\in \R_F$ are fixed. By choosing $c$ to be in an arbitrarily small neighborhood of $0$, we have that $y$ is not isolated. Hence $N_1(B)=0$.
	
	Suppose $\det(B-I)\neq 0$. There is a 1-1 correspondence between the set of fixed points of 
	$B$ and  the set $\Z_F^d/(B-I)\Z_F^d$. Since $\Z_F$ is a PID, we obtain the Smith Normal Form of $B-I$ that is a diagonal matrix with entries $b_1,\ldots,b_d\in\Z_F\setminus\{0\}$
	and a $\Z_F$-basis $x_1,\ldots,x_d$ of $\Z_F^d$ so that $b_1x_1,\ldots,b_dx_d$
	is a $\Z_F$-basis of $(B-I)\Z_F$. Therefore the number of fixed points of $B$ is:
	$$\prod_{i=1}^d \card (\Z_F/b_i\Z_F)=\prod_{i=1}^d\vert b_i\vert=\vert \det(B-I)\vert.$$	
	\end{proof}
	
	We fix once and for all a finite extension $K$ of $\R_F$ containing all the eigenvalues of $A$ and let $\delta$ be the inertia degree of $K/\R_F$. 
	For each $\mu_i$ in the  (possibly empty) 
	multiset $\{\mu_1,\ldots,\mu_M\}$ of eigenvalues of $A$ that are roots of unity, we 
	have the decomposition:
	$$\mu_{i}=\mu_{i,(0)}+\mu_{i,(1)}$$
	with $\mu_{i,(0)}\in\GF(q^{\delta})^*$ and $\mu_{i,(1)}\in \fp_K$ as in \eqref{eq:alpha0 and alpha1}; in fact $\mu_{i,(1)}=0$ since $\mu_i$ is a root of unity. Likewise, for each $\eta_{i}$ in 
	the (possibly empty) multiset 
	$\{\eta_1,\ldots,\eta_N\}$, we have:
	$$\eta_i=\eta_{i,(0)}+\eta_{i,(1)}$$
	with $\eta_{i,(0)}\in\GF(q^\delta)^*$ and $\eta_{i,(1)}\in \fp_K\setminus\{0\}$.
	\begin{proposition}\label{prop:main formulas}
	Let $v_p$ denote the $p$-adic valuation on $\Z$. Recall that the orders of
	$\mu_{i,(0)}$ and $\eta_{j,(0)}$ in $\GF(q^{\delta})^*$ are respectively denoted $m_i$ and $n_j$
	for $1\leq i\leq M$ and $1\leq j\leq N$; each of the $m_i$'s and $n_j$'s is coprime to $p$. Let $k$ be a positive integer, we have:
	\begin{itemize}
		\item [(i)] For $1\leq i\leq M$, $\vert\mu_i^k-1\vert=
\left\{
	\begin{array}{ll}
		0  & \mbox{if } k\equiv 0\bmod m_i \\
		1 & \mbox{otherwise}
	\end{array}
\right.$.

	\item [(ii)] For $1\leq j\leq N$, $\vert \eta_j^k-1\vert=\left\{
	\begin{array}{ll}
		\vert \eta_{j,(1)}\vert^{p^{v_p(k)}}  & \mbox{if } k\equiv 0\bmod n_j \\
		1 & \mbox{otherwise}
	\end{array}
\right.$
	
	\item [(iii)] $\displaystyle N_k(A)=\vert \det(A^k-I)\vert=r(A)^k\left(\prod_{i=1}^M a_{i,k}\prod_{j=1}^N b_{j,k}\right)^{p^{v_p(k)}}$ where
	$$a_{i,k}=\left\{
	\begin{array}{ll}
		0  & \mbox{if } k\equiv 0\bmod m_i \\
		1 & \mbox{otherwise}
	\end{array}
\right.\ \text{and}\ b_{j,k}=\left\{
	\begin{array}{ll}
		\vert \eta_{j,(1)}\vert  & \mbox{if } k\equiv 0\bmod n_j \\
		1 & \mbox{otherwise}
	\end{array}
\right.$$
	for $1\leq i\leq M$ and $1\leq j\leq N$.
	\end{itemize}
	\end{proposition}
	\begin{proof}
	Part (i) is easy: $\mu_i^k-1=\mu_{i,(0)}^k-1$ is an element of $\GF(q^\delta)$ and it is $0$ exactly when $k\equiv 0\bmod m_i$. For part (ii), when $k\not\equiv 0\bmod n_j$, we have:
	$$\eta_j^k-1\equiv \eta_{j,(0)}^k-1\not\equiv 0\bmod \fp_K,$$
	hence $\vert \eta_j^k-1\vert=1$. Now suppose $k\equiv 0\bmod n_j$ but $k\not\equiv 0\bmod p$, we have:
	$$\eta_j^k-1=(\eta_{j,(0)}+\eta_{j,(1)})^k-1=k\eta_{j,(0)}^{k-1}\eta_{j,(1)}+\sum_{\ell=2}^k \binom{k}{\ell}\eta_{j,(0)}^{k-\ell}\eta_{j,(1)}^{\ell}$$
	and since $\vert k\eta_{j,(0)}^{k-1}\eta_{j,(1)}\vert=\vert \eta_{j,(1)}\vert$ is strictly larger than the absolute value of each of the remaining terms, we have:
	$$\vert \eta_j^k-1\vert=\vert\eta_{j,(1)}\vert.$$
	Finally, suppose $k\equiv 0\bmod n_j$. Since $\gcd(n_j,p)=1$, we can write $k=k_0p^{v_p(k)}$ where
	$k_0\equiv 0\bmod n_j$ and $k_0\not\equiv 0\bmod p$. We have:
	$$\vert \eta_j^k-1\vert=\vert \eta_j^{k_0}-1\vert^{p^{v_p(k)}}=\vert\eta_{j,(1)}\vert^{p^{v_p(k)}}$$	
	and this finishes the proof of part (ii). Part (iii) follows from parts (i), (ii), and the definition of $r(A)$.
	\end{proof}
	
	\begin{proof}[Proof of Theorem~\ref{thm:zeta}]
	First, we prove part (a). We are given that for every $j\in\{1,\ldots,N\}$, there exists $i\in\{1,\ldots,M\}$ such that $m_i\mid n_j$. 
	
	Let $k\geq 1$. If $m_i\mid k$ for some $i$ then $N_k(A)=0$ by part (c) of Proposition~\ref{prop:main formulas}. If $m_i\nmid k$ for every $i\in\{1,\ldots,M\}$ then $n_j\nmid k$ for every $j\in\{1,\ldots,N\}$ thanks to the above assumption, then we have $N_k(A)=r(A)^k$ by Proposition~\ref{prop:main formulas}. Therefore $\displaystyle\sum_{k=1}^{\infty}\frac{N_k(A)}{k}z^k$ is equal to:
	\begin{align*}
	&\sum_{\substack{k\geq 1\\m_i\nmid k\text{ for }1\leq i\leq M}} \frac{N_k(A)}{k}z^k\\
	=&\sum_{\substack{k\geq 1\\m_i\nmid k\text{ for }1\leq i\leq M}} \frac{r(A)^k}{k}z^k\\
	=&\sum_{k\geq 1}\frac{r(A)^k}{k}z^k-\sum_{\substack{k\geq 1\\m_i\mid k\text{ for some }1\leq i\leq M}} \frac{r(A)^k}{k}z^k\\
	=&-\log(1-r(A)z)\\
	&-\sum_{\ell=1}^{M}\sum_{1\leq i_1<\ldots<i_{\ell}\leq M}(-1)^{\ell-1}\sum_{\substack{k\geq 1\\ \lcm(m_{i_1},\ldots,m_{i_{\ell}})\mid k}}\frac{r(A)^k}{k}z^k\\
	=&-\log(1-r(A)z)\\
	 &+\sum_{\ell=1}^{M}\sum_{1\leq i_1<\ldots<i_{\ell}\leq M}\frac{(-1)^{\ell+1}}{\lcm(m_{i_1},\ldots,m_{i_{\ell}})}\log\left(1-(r(A)z)^{\lcm(m_{i_1},\ldots,m_{i_{\ell}})}\right)
	\end{align*}
	where the third ``$=$'' follows from the inclusion-exclusion principle. This finishes the proof of part (a).
	
	For part (b), without loss of generality, we assume that $m_i\nmid n_1$ for $1\leq i\leq M$.
	Put
	$$f(z):=\sum_{k=1}^{\infty}N_k(A)z^k.$$
	Proposition~\ref{prop:main formulas} gives that 
	$\vert N_k(A)\vert\leq r(A)^k$, hence $f$ is convergent in the disk
	of radius $1/r(A)$. Assume that $f$ is D-finite and we arrive at a contradiction. Consider
	\begin{equation}\label{eq:c_k}
		c_k:=\frac{N_k(A)}{r(A)^k}\ \text{for $k=1,2,\ldots$}
	\end{equation}
	then the series
	$$\sum_{k=1}^\infty c_kz^k=f(z/r(A))$$
	is D-finite. Let $\tau$ denote the ramification index of $K/\R_F$, then each $\vert\eta_{j,(1)}\vert$ has the form $\displaystyle\frac{1}{q^{d_j/\tau}}$ where $d_j$ is a positive integer \cite[p.~150]{Neu99_AN}. Combining this with \eqref{eq:c_k} and Proposition~\ref{prop:main formulas}, we have that the $c_k$'s belong to the number field $E:=\Q(p^{1/\tau})$. Let $\vert\cdot\vert_p$ denote the $p$-adic absolute value on $\Q$, then $\vert\cdot\vert_p$ extends uniquely to an absolute value on $E$ since there is only one prime ideal of the ring of integers of $E$ lying above $p$. Put:
	$$Q=\prod_{1\leq j\leq N} \vert\eta_{j,(1)}\vert\ \text{and}\ Q_1=\prod_{\substack{1\leq j\leq N\\ n_j\mid n_1}}\vert \eta_{j,(1)}\vert.$$
	Since both $Q$ and $Q_1$ are powers of $1/q^{1/\tau}$ with positive integer exponents, we have:
	\begin{equation}\label{eq:QpQ1p}
	\vert Q\vert_p,\vert Q_1\vert_p >1.
	\end{equation}
	Since $m_i\nmid n_1$ for every $i$, Proposition~\ref{prop:main formulas}  and \eqref{eq:c_k}
	yield:
	\begin{equation}\label{eq:cn1pell}
	c_{n_1p^{\ell}}=Q_1^{p^{\ell}}\ \text{for every integer $\ell\geq 0$.}
	\end{equation}
	On the other hand, Proposition~\ref{prop:main formulas} and \eqref{eq:c_k} also yield:
	\begin{equation}\label{eq:ckp upper bound}
	\vert c_k\vert_p\leq \vert Q\vert_p^{p^{v_p(k)}}\ \text{for every integer $k>1$.}
	\end{equation}
	
	The idea to finish the proof is as follows. D-finiteness of the series 
	$\displaystyle\sum_{k=1}^{\infty} c_kz^k$ implies a strong restriction on the 
	``growth'' of the coefficients $c_k$'s at least through a recurrence relation satisfied by the $c_k$'s. This growth could be in terms of local data such as  absolute values of the $c_k$'s or global data such as Weil heights of the $c_k$'s \cite{BNZ20_DF}. It is indeed the $\vert c_k\vert_p$'s that will give us the desired contradiction. The key observation is that when $\ell$ is large $\displaystyle \vert c_{n_1p^\ell}\vert_p=\vert Q_1\vert_p^{p^\ell}$ is exponential in $p^{\ell}$ thanks to \eqref{eq:QpQ1p}
	and \eqref{eq:cn1pell} while the ``nearby'' coefficients $c_{n_1p^{\ell}-n}$ for a \emph{bounded} positive integer $n$ have small $p$-adic absolute values thanks to \eqref{eq:ckp upper bound} since
	$v_p(n_1p^{\ell}-n)$ is small compared to $\ell$.
	
	Since $\displaystyle\sum_{k=1}^\infty c_kz^k\in E[[z]]$ is D-finite, there exist a positive integer $s$ and polynomials $P_0(z),\ldots,P_s(z)\in E[z]$ such that $P_0\neq 0$ and
	\begin{equation}\label{eq:P recursive}
	P_0(k)c_k+P_1(k)c_{k-1}+\ldots+P_s(k)c_{k-s}=0
	\end{equation}
	for all sufficiently large $k$ \cite{Sta80_DF}. In the following $\ell$ denotes a large positive integer and the implied constants in the various estimates are independent of $\ell$. Consider $k=n_1p^{\ell}$, then the highest power of $p$ dividing any of the $k-i=n_1p^{\ell}-i$ for $1\leq i\leq s$ is at most the largest power of $p$ in $\{1,2,\ldots,s\}$. Combining this with \eqref{eq:ckp upper bound}, we have:
	\begin{equation}\label{eq:Picn1pell-i}
	\vert P_i(n_1p^{\ell})c_{n_1p^{\ell}-i}\vert_p \ll 1\ \text{for $1\leq i\leq s$.}
	\end{equation}
	
	 Now \eqref{eq:cn1pell}, \eqref{eq:P recursive}, and \eqref{eq:Picn1pell-i} imply:
	 \begin{equation}\label{eq:exponentially small p-adic}
	 \vert P_0(n_1p^{\ell})\vert_p\ll \vert Q_1\vert_p^{-p^{\ell}}. 
	 \end{equation}
	This means for the infinitely many positive integers $k$ of the form $n_1p^{\ell}$, we have that $\vert P_0(k)\vert_p$ is exponentially small in $k$. This implies that $k$ is unusually close to a root of $P_0$ with respect to the $p$-adic absolute value. One can use the product formula to arrive at a contradiction, as follows. 
	
	Let $M_E=M_E^{0}\cup M_E^{\infty}$ be the set of all places of $E$ where $M_E^0$ consists of the finite places and $M_E^{\infty}$ denotes the set of all the infinite places \cite[Chapter~1]{BG06_HI}. For every $w\in M_E$, we normalize $\vert\cdot\vert_w$ as in \cite[Chapter~1]{BG06_HI} and the product formula holds. We still use $p$ to denote the only place of $E$ lying above $p$
	and the above $\vert\cdot\vert_p$ has already been normalized according to \cite[Chapter~1]{BG06_HI}. We have:
	\begin{equation}\label{eq:all the w that are not p}
	\prod_{w\in M_K^{\infty}} \vert P_0(n_1p^{\ell})\vert_w \ll (n_1p^{\ell})^{\deg(P_0)}\ \text{and}\ 
	\prod_{w\in M_K^0\setminus\{p\}} \vert P_0(n_1p^{\ell})\vert_w\ll 1.
	\end{equation}
	When $\ell$ is sufficiently large and $P_0(n_1p^{\ell})\neq 0$, we have that \eqref{eq:QpQ1p},  \eqref{eq:exponentially small p-adic} and \eqref{eq:all the w that are not p} contradict the product formula:
	$$\prod_{w\in M_K}\vert P_0(n_1p^{\ell})\vert_w=1$$
	and this finishes the proof that $\displaystyle f(z)=\sum_{k=1}^\infty N_k(A)z^k$ is not D-finite. The transcendence of $\zeta_A(z)$ follows immediately: if $\zeta_A(z)$ were algebraic then 
	$\displaystyle f(z)=z\frac{\zeta_A'(z)}{\zeta_A(z)}$ would be algebraic and hence D-finite, see Remark~\ref{rem:algebraic implies D-finite}.
	\end{proof}

	\bibliographystyle{amsalpha}
	\bibliography{pTori} 	

\def\cprime{$'$} \def\cprime{$'$} \def\cprime{$'$} \def\cprime{$'$}
\providecommand{\bysame}{\leavevmode\hbox to3em{\hrulefill}\thinspace}
\providecommand{\MR}{\relax\ifhmode\unskip\space\fi MR }
\providecommand{\MRhref}[2]{%
  \href{http://www.ams.org/mathscinet-getitem?mr=#1}{#2}
}
\providecommand{\href}[2]{#2}
\begin{thebibliography}{BMW14}

\bibitem[AM65]{AM65_OP}
M.~Artin and B.~Mazur, \emph{On periodic points}, Ann. of Math. (2) \textbf{81}
  (1965), 82--99.

\bibitem[BG06]{BG06_HI}
E.~Bombieri and W.~Gubler, \emph{Heights in {D}iophantine geometry}, New
  Mathematical Monographs, vol.~4, Cambridge University Press, Cambridge, 2006.

\bibitem[BGNS]{BGNS22_AG}
J.~P. Bell, K.~Gunn, K.~D. Nguyen, and J.~C. Saunders, \emph{A general
  criterion for the {P}\'olya-{C}arlson dichotomy and application}, available
  on the arXiv, 2022.

\bibitem[BK83]{BK83_OL}
M.~Brin and A.~Katok, \emph{On local entropy}, Geometric dynamics (Rio de
  Janeiro, 1981),, Lecture Notes in Math., no. 1007, Springer-Verlag, 1983,
  pp.~30--38.

\bibitem[BL16]{BL16_AW}
V.~Bergelson and A.~Leibman, \emph{A {W}eyl-type equidistribution theorem in
  finite characteristic}, Adv. Math. \textbf{289} (2016), 928--950.

\bibitem[BL19]{BL19_LA}
P.-Y. Bienvenu and T.-H. Le, \emph{Linear and quadratic uniformity of the
  {M}\"obius function over $\mathbb{F}_q[t]$}, Mathematika \textbf{65} (2019),
  505--529.

\bibitem[BLP10]{BLP10_AN}
M.~Baake, E.~Lau, and V.~Paskunas, \emph{A note on the dynamical zeta function
  of general toral endomorphisms}, Monatsh. Math. \textbf{161} (2010), 33--42.

\bibitem[BMW14]{BMW14_TA}
J.~Bell, R.~Miles, and T.~Ward, \emph{Towards a p\'olya–carlson dichotomy for
  algebraic dynamics}, Indag. Math. (N.S.) \textbf{25} (2014), 652--668.

\bibitem[BNZ]{BNZ22_DF2}
J.~P. Bell, K.~D. Nguyen, and U.~Zannier, \emph{D-finiteness, rationality, and
  height {II}: lower bounds over a set of positive density}, arXiv:2205.02145.

\bibitem[BNZ20]{BNZ20_DF}
\bysame, \emph{D-finiteness, rationality, and height}, Trans. Amer. Math. Soc.
  \textbf{373} (2020), 4889--4906.

\bibitem[Bri12]{Bri12_TO}
A.~Bridy, \emph{Transcendence of the {A}rtin-{M}azur zeta function for
  polynomial maps of $\mathbb{A}^1(\bar{F}_p)$}, Acta Arith. \textbf{156}
  (2012), 293--300.

\bibitem[Bri16]{Bri16_TA}
\bysame, \emph{The {A}rtin-{M}azur zeta function of a dynamically affine
  rational map in positive characteristic}, J. Th\'eor. Nombres Bordeaux
  \textbf{28} (2016), no.~2, 301--324.

\bibitem[Car21]{Car21_U}
F.~Carlson, \emph{{\"U}ber ganzwertige funktionen}, Math. Z. \textbf{11}
  (1921), 1--23.

\bibitem[DF04]{DF04_AA}
D.~S. Dummit and R.~M. Foote, \emph{Abstract algebra}, third ed., Wiley, 2004.

\bibitem[Dwo60]{Dwo60_OT}
B.~Dwork, \emph{On the rationality of the zeta function of an algebraic
  variety}, Amer. J. Math. \textbf{82} (1960), 631--648.

\bibitem[Guc70]{Guc70_AA}
J.~Guckenheimer, \emph{Axiom {A}+{N}o {Cycles} {$\Longrightarrow \zeta_f(t)$}
  {Rational}}, Bull. Amer. Math. Soc. \textbf{76} (1970), 592--594.

\bibitem[Hay65]{Hay65_TD}
D.~R. Hayes, \emph{The distribution of irreducibles in {$GF[q,x]$}}, Trans.
  Amer. Math. Soc. \textbf{117} (1965), 101--127.

\bibitem[Hin94]{Hin94_ZF}
A.~Hinkkanen, \emph{Zeta functions of rational functions are rational}, Ann.
  Acad. Sci. Fenn. Ser. AI Math. \textbf{19} (1994), 3--10.

\bibitem[LW10]{LW10_WP}
Y.-R. Liu and T.~Wooley, \emph{Waring's problem in function fields}, J. reine
  angew. Math. \textbf{638} (2010), 1--67.

\bibitem[Man71]{Man71_AA}
A.~Manning, \emph{Axiom {A} diffeomorphisms have rational zeta functions},
  Bull. Lond. Math. Soc. \textbf{3} (1971), 215--220.

\bibitem[Neu99]{Neu99_AN}
J.~Neukirch, \emph{{A}lgebraic {N}umber {T}heory}, Grundlehren der
  mathematischen Wissenschaften, vol. 322, Springer-Verlag, 1999, Translated
  from the German by N. Schappacher.

\bibitem[Por18]{Por18_AN}
S.~Porritt, \emph{A note on exponential-{M}{\"o}bius sums over
  $\mathbb{F}_q[t]$}, Finite Fields Appl. \textbf{51} (2018), 298--305.

\bibitem[P{\'o}y28]{Pol28_U}
G.~P{\'o}ya, \emph{{\"U}ber gewisse notwendige {D}eterminantenkriterien f{\"u}r
  {F}ortsetzbarkeit einer {P}otenzreihe}, Math. Ann. \textbf{99} (1928),
  687--706.

\bibitem[Sta80]{Sta80_DF}
R.~Stanley, \emph{Differentiably finite power series}, European J. Combin.
  \textbf{1} (1980), 175--188.

\bibitem[VO16]{VO16_FO}
M.~Viana and K.~Oliveira, \emph{Foundations of {E}rgodic {T}heory}, Cambridge
  studies in advanced mathematics, vol. 151, Cambridge University Press,
  Cambridge, 2016.

\bibitem[Wal82]{Wal82_AI}
P.~Walters, \emph{An {I}ntroduction to {E}rgodic {T}heory}, Graduate Texts in
  Mathematics, vol.~79, Springer-Verlag, New York, 1982.

\bibitem[Wei49]{Wei49_NO}
A.~Weil, \emph{Numbers of solutions of equations in finite fields}, Bull. Amer.
  Math. Soc. \textbf{55} (1949), 497--508.

\end{thebibliography}
\end{document}